\documentclass[reqno, 11pt]{amsart}
\usepackage{amsmath,mathtools}
 \usepackage{amssymb}
\usepackage{amsthm}
\usepackage{amsmath}
\usepackage{times}
\usepackage{latexsym}
\usepackage[mathscr]{eucal}

\numberwithin{equation}{section}
 
  \newtheorem{theorem}{Theorem}[section]
  \newtheorem{proposition}[theorem]{Proposition}
  \newtheorem{lemma}[theorem]{Lemma}
  \newtheorem{corollary}[theorem]{Corollary}

  \newtheorem{remark}[theorem]{Remark}

  \newtheorem{example}[theorem]{Example}

\title[Remarks on screen integrable null hypersurfaces]{Remarks on screen integrable null hypersurfaces  in Lorentzian manifolds}
\author[Samuel Ssekajja]{Samuel Ssekajja*}
\newcommand{\acr}{\newline\indent}
\address{\llap{*\,} School of Mathematics\acr
 University of Witwatersrand\acr
 Private Bag 3, Wits 20150\acr
South Africa}
\email{samuel.ssekajja@wits.ac.za} 
\thanks{}

 
\begin{document}
\begin{abstract}
 In the present paper, we show that the geometry of a screen integrable null hypersurface can be generated from an isometric immersion of a leaf of its screen distribution into the ambient space. We prove, under certain geometric conditions,  that such immersions are contained in semi-Euclidean spheres or hyperbolic spaces, and the underlying null hypersurfaces are necessarily umbilic and screen totally umbilic. Where necessary, examples have been given to illustrate the main ideas. 
\end{abstract}
\maketitle
\section{Introduction} 

 A study of null submanifolds in semi-Riemannian manifolds was introduced by Duggal-Benjancu \cite{db} and later updated by Duggal-Sahin \cite{ds2}. In the above books, the authors laid a foundation for research on null geometry by constracting their structural equations, among other results.  In fact, they introduced a non-degenerate screen distribution to construct a null transversal vector bundle which is non-intersecting to its null tangent bundle and developed local geometry of null curves, hypersurfaces and submanifolds. Other pioneers of the theory include D. N. Kupeli \cite{kup}--whose approach is purely intrinsic compared to that of \cite{db,ds2}. Since then, many researchers including but not limited to; \cite{antii,moh,Jin, Jin11}, have researched on null submanifolds and many interesting results have been obtained. Null hypersurfaces appears in general relativity as models of different types of black hole horizons (see \cite{db,ds2} for details) and their theory is quite fundamental to modern mathematical physics.

Among the most studied null hypersurfaces are those with an integrable screen distribution, and they are commonly known us {\it screen integrable null hypersurfaces}. They include the well-known screen  conformal ones, among others. It was shown in \cite{jinduga}, that all screen integrable null hypersurfaces are locally isometric to $\mathcal{C}_{\xi}\times M^{*}$, where $\mathcal{C}_{\xi}$ is a null curve tangent to the normal bundle of the hypersurface and $M^{*}$ is a leaf of its screen distribution. In particular, \cite{db} proves that a null cone of an $(n+2)$-dimensional Lorentzian space $\mathbb{R}_{1}^{n+2}$ is screen conformal, satisfying the above structure, with $M^{*}\cong \mathbb{S}^{n}$. Under some geometric conditions on the ambient space, Duggal-Sahin \cite{ds2} also proves that a screen conformal Einstein null hypersurface  is locally a triple product  $\mathcal{C}_{\xi}\times M^{*}_{\alpha}\times M^{*}_{\beta}$, where $M^{*}_{\alpha}$ and $M^{*}_{\beta}$ are some leaves of its screen distribution (see Theorem 2.5.17 of \cite{ds2} for more details). In the book \cite{db}, Duggal and Bejancu tries to understand the geometry of a null hypersurface $M$ from a leaf $M^{*}$ of its screen distribution as an immersion in the ambient space. He,  in fact, shows that an umbilic leaf in the ambient space implies that the underlying null hypersurface is umbilic too (see Proposition 5.1 of \cite[p. 107]{db}). A natural question then arises; {\it Which other geometric information, about the null hypersurface, can be derived from the geometry of an isometric immersion of a leaf of its screen distribution into the umbient space?} 

The main aim of this paper is to give some solutions to the above question by studying null hypersurfaces of Lorentzian spaces. Consequently, we prove two main theorems in that line; Theorems \ref{theorem1} and \ref{theorem2}.  The paper is arranged as follows; In Section \ref{pre}, we quote some basic notions needed in the rest of the paper. In  Section \ref{theorems}, we prove several characterization results.

\section{Preliminaries} \label{pre}

Suppose $M$ is an $(n+1)$-dimensional smooth manifold and $F : M \longrightarrow \overline{M}$  a smooth mapping such that each point $x\in M$ has an open neighborhood $\mathcal{U}$ for which $F$ restricted to $\mathcal{U}$ is one-to-one and $F^{-1}:F(\mathcal{U})\longrightarrow M$ are smooth. Then, we say that $F(M)$ is an immersed hypersurface of $\overline{M}$. If this condition globally holds, then $F(M)$ is called an embedded hypersurface of $\overline{M}$, which we assume in this paper. The embedded hypersurface has a natural manifold structure inherited from the manifold structure on $\overline{M}$ via the embedding mapping. At each point $F(x)$ of $F(M)$, the tangent space is naturally identified with an $(n+1)$-dimensional subspace $T_{F(x)}M$ of the tangent space $T_{F(x)}\overline{M}$. The embedding $F$ induces, in general, a symmetric tensor field, say $g$, on $F(M)$ such that  $g(X,Y)|_{x}=\overline{g}(F_{*}X,F_{*}Y)|_{F(x)}$, for all $X,Y\in T_{x}M$. Here, $F_{*}$ is the differential map of $F$ defined by  $F_{*}:T_{x}M\longrightarrow T_{F(x)}\overline{M}$ and $(F_{x}X)\omega=X(\omega\circ F)$, for an arbitrary smooth function $\omega$ in a neighborhood of $F(x)$ of $F(M)$. Henceforth, we write $M$ and $x$ instead of $F(M)$ and $F(x)$. Due to the causal character of three categories (spacelike, timelike and lightlike) of the vector fields of $\overline{M}$, there are three types of hypersurfaces $M$, namely, {\it Riemannian}, {\it semi-Riemannian} and {\it null} (or {\it lightlike}) and $g$ is a non-degenerate or a degenerate symmetric tensor field on $M$ according as $M$ is of the first two types and of the third type, respectively. The geometry of Riemannian or semi-Riemannian hypersurfaces is well-known and has received a considerable attention, for example see \cite{oneil} and many more references cited therein. In the present paper, we focus on null hypersurfaces using the approach of Duggal-Bejancu \cite{db}. 
   
   Now, let $g$ be degenerate on $M$. Then, there exists a {\it nonzero} vector field $\xi$ on $M$ such that $g(\xi,X)=0$, for all $X\in \Gamma(TM)$. The {\it radical} or the {\it null space} \cite[p. 53]{oneil} of $T_{x}M$, at each point $x\in M$, is a subspace $\mathrm{Rad}\, T_{x}M$ defined by $\mathrm{Rad}\, T_{x}M=\{\xi \in T_{x}M:g_{x}(\xi, X)=0,\;\;\forall \, X\in T_{x}M\}$, whose dimension is called the {\it nullity degree} of $g$ and $M$ is called a {\it null hypersurface} of $\overline{M}$. It follows  that $T_{x}M^{\perp}$  is also null and satisfy $
   \mathrm{Rad}\, T_{x}M=T_{x}M\,\cap \,T_{x}M^{\perp}
   $. For a hypersurface $M$ $\dim (T_{x}M^{\perp}) = 1$, implies that $\dim(\mathrm{Rad}\,T_{x}M)=1$ and $\mathrm{Rad}\,T_{x}M=T_{x}M^{\perp}$. We call $\mathrm{Rad}\, TM$ a radical (null) distribution of $M$. Thus, for a null hypersurface $M$, $TM$ and $TM^{\perp}$ have a nontrivial intersection and their sum is not the whole of tangent bundle space $T\overline{M}$ . In other words, a vector of $T_{x}\overline{M}$ cannot be decomposed uniquely into a component tangent to $T_{x}M$ and a component of $T_{x}M^{\perp}$. Therefore, the standard text-book definition of the second fundamental form and the Gauss-Weingarten formulas do not work, in the usual way, for the null case.

  To overcome the above difficulty, Duggal-Bejancu \cite{db} introduced an approach to null geometry, which we follow in this paper. The approach  consists of fixing, on the null hypersurface, a geometric data formed by a null section  and  a  {\it screen  distribution}.   By screen  distributionon of $M$,  we  mean  a complementary bundle of $TM^{\perp}$ in $TM$.  It is then a rank $n$ non-degenerate distribution over $M$.  In fact, there are infinitely many possibilities of choices for such a distribution provided the hypersurface $M$ is  paracompact, but each of them is canonically isomorphic to the factor vector bundle $TM/TM^{\perp}$ \cite{kup}.  We denote by $S(TM)$ the screen distribution over $M$. Then we have the decompostion $TM=S(TM)\perp TM^{\perp}$, where $\perp$ denotes the orthogonal direct sum.   From \cite{db} or \cite{ds2}, it is known that for a null hypersurface equipped with a screen distribution, there exists a unique rank $1$ vector subbundle $tr(TM)$ of $T\overline{M}$ over $M$, such that for any non-zero section $\xi$ of $TM^{\perp}$ on a coordinate neighborhood $\mathcal{U}\subset M$, there exists a unique section $N$ of $tr(TM)$ on $\mathcal{U}$ satisfying 
$
	\overline{g}(N,\xi)=1$, $\overline{g}(N,N)=\overline{g}(N,W)=0$, for all  $W\in \Gamma(S(TM)|_{\mathcal{U}}).
$
It then follows that $T\overline{M}|_{M}=S(TM)\perp \{TM^{\perp}\oplus tr(TM)\}=TM\oplus tr(TM),\label{v11}$ where $\oplus$ denote the direct (non-orthogal) sum. We call $tr(TM)$ a (null) {\it transversal} vector bundle along $M$.  Throughout the paper, all manifolds are supposed to be paracompact and smooth. We denote by $\mathcal{F}(M)$ the algebra of differentiable functions on $M$ and by $\Gamma(E)$ the $\mathcal{F}(M)$-module of differentiabale sections of a vector bundle $E$ over $M$. We also assume that all associated structures are smooth.

  Let $\nabla$ and $\nabla^{*}$ denote the induced connections on $M$ and $S(TM)$, respectively, and $P$ be the projection of $TM$ onto $S(TM)$, then the local Gauss-Weingarten equations of $M$ and $S(TM)$ are the following \cite{db}
\begin{align}
 \overline{\nabla}_{X}Y&=\nabla_{X}Y+B(X,Y)N,\label{v12}\\
 \overline{\nabla}_{X}N&=-A_{N}X+\tau(X)N,\label{v13}\\
  \nabla_{X}PY&= \nabla^{*}_{X}PY + C(X,PY)\xi,\label{v14}\\
  \nabla_{X}\xi &=-A^{*}_{\xi}X -\tau(X) \xi,\;\;\; A_{\xi}^{*}\xi=0,\label{v15}
 \end{align}
 for all $X,Y\in\Gamma(TM)$, $\xi\in\Gamma(TM^{\perp})$ and $N\in\Gamma(tr(T M))$. In the above setting, $B$ is the local second fundamental form of $M$ and $C$  is the local second fundamental form on $S(TM)$.  $A_{N}$ and $A^{*}_{\xi}$ are the shape operators on $TM$ and $S(TM)$ respectively, while $\tau$ is a 1-form on $TM$. The above shape operators are related to their local fundamental forms by $g(A^{*}_{\xi}X,Y) =B(X,Y)$, $g(A_{N}X,PY) = C(X,PY)$, for any $X,Y\in \Gamma(TM)$. Moreover, $\overline{g}(A^{*}_{\xi}X,N)=0$ and $\overline{g}(A_{N} X,N)=0$, for all $ X\in\Gamma(TM)$. From the above relations, we notice that $A_{\xi}^{*}$ and $A_{N}$ are both screen-valued operators.  
 
 The null hypersurface $M$ is said to be \textit{totally umbilic} \cite{db} if $B=\rho \otimes g$, where $\rho$ is a smooth function on a coordinate neighborhood  $\mathcal{U}\subset TM$. In case $\rho=0$, we say that $M$ is \textit{totally geodesic}. In the same line, $M$ is called \textit{screen totally umbilic} if $C=\varrho \otimes g$,  where $\varrho$ is a smooth function on a coordinate neighborhood $\mathcal{U}\subset TM$. When $\varrho=0$, we say that $M$ is \textit{screen totally geodesic}. The {\it mean curvature vector} $H$ of a null hypersurface is transverssal  to $M$, and given by $H=\frac{1}{n}(\mathrm{trace}_{S(TM)}B)N=\frac{1}{n}(\mathrm{trace}_{S(TM)}A_{\xi}^{*})N$. We say that $M$ is a {\it minimal} null hypersurface if $H=0$. More precisely, $M$ is minimal if $\mathrm{trace}_{S(TM)}A_{\xi}^{*}=0$ (see \cite{db,ds2} for more details and examples).
 
 Let $\vartheta=\overline{g}(N,\boldsymbol{\cdot})$ be a 1-form metrically equivalent to $N$ defined on $\overline{M}$. Take $\eta=i^{*}\vartheta$ to be its restriction on $M$, where $i:M\rightarrow \overline{M}$ is the inclusion map. Then it is easy to show that $(\nabla_{X}g)(Y,Z)=B(X,Y)\eta(Z)+B(X,Z)\eta(Y)$, for all $X,Y,Z\in \Gamma(TM)$. Consequently,  $\nabla$ is generally \textit {not} a metric connection with respect to $g$. However, the induced connection $\nabla^{*}$ on $S(TM)$ is a metric connection. Denote by $\overline{R}$ the curvature tensor of the connection $\overline{\nabla}$. Using the Gauss-Weingarten formulae (\ref{v12})-(\ref{v15}), we obtain the following  curvature relations (see details in \cite{db,ds2}).
\begin{align}
\overline{g}(\overline{R}(X,Y)\xi,N)=&C(Y,A_{\xi}^{*}X)-C(X,A_{\xi}^{*}Y)-2d\tau(X,Y),\label{v35}
\end{align}
where $2d\tau(X,Y)=X\tau(Y)-Y\tau(X)-\tau([X,Y])$,
for all $X,Y\in \Gamma(TM)|_{\mathcal{U}}$, $\xi\in \Gamma(TM^{\perp})$ and $N\in \Gamma(tr(TM))$. 

Suppose $\pi$ is a non-degenerate plane of $T_{p}\overline{M}$, for $p\in \overline{M}$. Then, the associated matrix $G_{p}$ of $\overline{g}_{p}$, with respect to an arbitrary basis $\{u, v\}$, is of rank 2 given by (1.2.15) of \cite[p. 16]{ds2}. Define a real number $K(\pi) = K_{p}(u, v) = \overline{R}(u, v, v, u)$, where $\overline{R}(u, v, v, u)$ is the 4-linear mapping on $T_{p}\overline{M}$ by the curvature tensor. The smooth function $K$, which assigns to each non-degenerate tangent plane $\pi$ the real number $K(\pi)$ is called the {\it sectional curvature} of $\overline{M}$, which is independent of the basis $\{u, v\}$. If $K$ is a constant $c$ at every point of $p\in \overline{M}$ then $\overline{M}$ is of constant sectional curvature $c$, denote by $\overline{M}(c)$, whose curvature tensor field $\overline{R}$  is given by $\overline{R}(X,Y)Z=c\{\overline{g}(Y,Z)X-\overline{g}(X,Z)Y\}$, for any $X,Y,Z\in \Gamma(T\overline{M})$ (see \cite{oneil} for details). In particular, if $K = 0$, then $\overline{M}$ is called a {\it flat manifold} for which $\overline{R}=0$.

\section{Geometry of $(M,g)$ from that of a leaf of $S(TM)$}\label{theorems}

Assume that $(M,g)$ is a screen integrable null hypersurface of a {\it Lorentzian manifold} $(\overline{M},\overline{g})$. Let $M^{*}$ be a (Riemannian)leaf  of its screen distribution $S(TM)$. Let $f:M^{*}\longrightarrow \overline{M}$ be an isometric immersion of $M^{*}$ in $\overline{M}$, as a codimension 2 nondegenerate submanifold, then (\ref{v12}) and (\ref{v14}) gives the Gauss formula of $M^{*}$ (in $\overline{M}$) as 
\begin{align}\label{t2}
	\overline{\nabla}_{X}Y=\nabla^{*}_{X}Y+C(X,Y)\xi+B(X,Y)N,\;\;\;\forall\, X,Y\in \Gamma(TM^{*}).
\end{align}
It is obvious from (\ref{t2}) that the second fundamental form $h^{*}$ of $M^{*}$, as a submanifold of $\overline{M}$, is given by $h^{*}(X,Y)=C(X,Y)\xi+B(X,Y)N$.  Next, denote by $\nabla^{*\perp}$ the normal connection on the normal bundle  $TM^{*\perp}$. Then, the Weingarten formula for $M^{*}$ is given by 
\begin{align}\label{t3}
	\overline{\nabla}_{X}V=-A_{V}X+\nabla^{*\perp}_{X}V,\;\;\;\forall\, X\in \Gamma(TM^{*}),\;\;V\in \Gamma(TM^{*\perp}),
\end{align}
where $A_{V}$ denotes the shape operator of $M^{*}$. Since $TM^{*\perp}=TM^{\perp}\oplus tr(TM)$, we let $V=a\xi +bN$, such that $a,b\ne 0$. Then, it is easy to see that $W=a\xi-bN$ is another vector field of $TM^{*\perp}$ which is orthogonal to $V$. {\it From now on, we consider $TM^{*\perp}$ spanned by $V$ and $W$}. Putting all the above into account, we can express the shape operator $A_{V}$ of $M^{*}$ in terms of the shape operators $A_{\xi}^{*}$ and $A_{N}$ as follows.
\begin{lemma}\label{lemm}
	Let $f:M^{*}\longrightarrow \overline{M}$ an isometric immersion such that (\ref{t2}) and (\ref{t3}) holds. Then, the shape operator of $M^{*}$ satisfies  $A_{V}=aA_{\xi}^{*}+bA_{N}$, where $V=a\xi+bN$.
\end{lemma}
\begin{proof}
	Taking the $\overline{g}$-product of (\ref{t3}) with $Y\in \Gamma(TM^{*})$ and using the fact that $\overline{\nabla}$ is a metric connection, we get $g(A_{V}X,Y)=\overline{g}(V,\overline{\nabla}_{X}Y)$. Then, applying (\ref{t2}) to the last relation and the fact that $M^{*}$ is nondegenerate, we get the desired result.
\end{proof}
\noindent Let $\{V,W\}$ be an orthonormal basis of $T_{x}M^{*\perp}$ at $x\in M^{*}$. Then, the {\it mean curvature vector} of a leaf $M^{*}$ in $\overline{M}$ is the vector  $H^{*}=\frac{1}{2}[(\mathrm{trace}A_{V})V+(\mathrm{trace}A_{W})W]$. We say that $M^{*}$ is {\it minimal} in $\overline{M}$ if $H^{*}$ vanishes. It then follows that $M^{*}$ is minimal if and only if $\mathrm{trace}A_{V}=0$ and $\mathrm{trace}A_{W}=0$. In view of Lemma \ref{lemm}, one can easily see that minimality of a leaf $M^{*}$ implies minimality of the underlying null hypersurface $(M,g)$. Let us consider the curvature tensor of the normal bundle $TM^{*\perp}$ as  $R^{*\perp}:T_{x}M^{*}\times T_{x}M^{*}\times T_{x}M^{*\perp}\longrightarrow T_{x}M^{*\perp}$, given by 
\begin{align}\label{t4}
	R^{*\perp}(X,Y)V=\nabla^{*\perp}_{X}\nabla^{*\perp}_{Y}V-\nabla^{*\perp}_{Y}\nabla^{*\perp}_{X}V-\nabla^{*\perp}_{[X,Y]}V,
	\end{align}
for any $X,Y\in \Gamma(TM^{*})$ and $V\in \Gamma(TM^{*\perp})$.
The importance of the 1-form $\tau$ in the study of null geometry has been shown in \cite{db} and \cite{ds2}. In fact, it has been shown that the Ricci tensor of a null submanfold is symmetric if and only if $\tau$ is closed, that is; $d\tau=0$. In what follows, we show that the normal curvature $R^{*\perp}$ of a leaf $M^{*}$ is directly linked to the 1-form $\tau$ of  (\ref{v13}).
 \begin{proposition}\label{pop1}
 	Let $(M,g)$ be a screen integrable null hypersurface of a Lorentzian manifold $(\overline{M}(c),\overline{g})$. Then, the normal curvature $R^{*\perp}$ of any leaf $M^{*}$ of $S(TM)$ satisfies 
 	\begin{align}
 	R^{*\perp}(X,Y)V&=\{C (X,A_{V}Y)-C(Y,A_{V}X)\}\xi \nonumber \\
 	&+\{B(X,A_{V}Y)-B(Y,A_{V}X)\}N\label{t6}\\
 	&=  2d\tau(X,Y)W,\label{t7}
 	\end{align}
 	for any vector fields $X,Y\in \Gamma(TM^{*})$ and $V,W\in \Gamma(TM^{*\perp})$.
 \end{proposition}
 \begin{proof}
 	A direct calculation using (\ref{t2}) and (\ref{t3}) leads to 
 	\begin{align}
 		\overline{R}(X,Y)V&=-\nabla^{*}_{X}A_{V}Y-C(X,A_{V}Y)\xi -B(X,A_{V}Y)N-A_{\nabla^{*\perp}_{Y}V}X\nonumber\\
 		&+\nabla_{X}^{*\perp}\nabla_{Y}^{*\perp}V+\nabla_{Y}^{*}A_{V}X+C(Y,A_{V}X)\xi+B(Y,A_{V}X)N\nonumber\\
 		&+A_{\nabla^{*\perp}_{X}V}Y-\nabla_{Y}^{*\perp}\nabla_{X}^{*\perp}V	+A_{V}[X,Y]-\nabla_{[X,Y]}^{*\perp}V\nonumber\\
 		&= -\nabla^{*}_{X}A_{V}Y+\nabla_{Y}^{*}A_{V}X-A_{\nabla^{*\perp}_{Y}V}X+A_{\nabla^{*\perp}_{X}V}Y+A_{V}[X,Y]\nonumber\\
 		&+R^{*\perp}(X,Y)V+	\{C (Y,A_{V}X)-C(X,A_{V}Y)\}\xi \nonumber\\
 		&+\{B(Y,A_{V}X)-B(X,A_{V}Y)\}N,\;\;\;\forall\, X,Y\in \Gamma(TM^{*})	\label{t5}
 	\end{align}
 	Since $\overline{M}$ is a space of constant curvature $c$, we have $\overline{R}(X,Y)V=0$, for any $X,Y\in \Gamma(TM^{*})$ and $V\in \Gamma(TM^{*\perp})$. Thus, (\ref{t5}) gives \begin{align}
 		R^{*\perp}(X,Y)V+&	\{C (Y,A_{V}X)-C(X,A_{V}Y)\}\xi \nonumber\\
 		&+\{B(Y,A_{V}X)-B(X,A_{V}Y)\}N=0, \label{t8}	
 	\end{align}
 	which proves (\ref{t6}). Next, applying Lemma \ref{lemm} to (\ref{t8}) and using the fact that $B(A^{*}_{\xi}X,Y)=B(X,A^{*}_{\xi}Y)$ and $C(A_{N}X,Y)=C(X,A_{N}Y)$, for all $X,Y\in \Gamma(S(TM))$, we get 
 	\begin{align}
 		R^{*\perp}(X,Y)V+&\{C(Y,A^{*}_{\xi}Y)-C(X,A^{*}_{\xi}X)\}a\xi\nonumber\\
 		&+\{B(Y,A_{N}X)-B(X,A_{N}Y)\}bN=0.\label{t9}
 	\end{align} 	 	
 	As $B(Y,A_{N}X)=g(A_{N}X,A^{*}_{\xi}Y )= C(X,A^{*}_{\xi}Y)$, (\ref{t9}) reduces to 
 	\begin{align}\label{t10}
 		R^{*\perp}(X,Y)V=\{C(Y,A^{*}_{\xi}Y)-C(X,A^{*}_{\xi}X)\}(a\xi-bN). 	\end{align}
 	Next, as $\overline{M}$ is a space of constant curvature $c$, we have $\overline{R}(X,Y)\xi=0$, for all $X,Y\in \Gamma(TM)$. Thus, in view of (\ref{v35}) and (\ref{t10}), we conclude that 	
 	\begin{align}\nonumber
 	R^{*\perp}(X,Y)V=2d\tau(X,Y)W,	
 	\end{align}
 	where $W=a\xi-bN$, which proves (\ref{t7}) and proof is completed.
 \end{proof}
\noindent The following follows directly from Proposition \ref{pop1}.
 	 \begin{corollary}\label{b1} 
 	 In view of Proposition \ref{pop1}, the  following are equivalent;
 	 	\begin{enumerate}
 	 		\item $d\tau$ vanishes on $S(TM)$.
 	 		\item $A_{\xi}^{*}\circ A_{N}=A_{N}\circ A_{\xi}^{*}$.
 	 		\item $R^{*\perp}(x)=0$.
 	 		\item The normal bundle of $M^{*}$ is parallel. 
 	 	\end{enumerate}	 	
 	 \end{corollary}
 	 	\begin{corollary}
 	 		If $(M,g)$  is totally umbilic or screen totally umbilic in $\overline{M}$, then $R^{*\perp}(x)=0$.
  	 	\end{corollary}
  	 \begin{remark}\label{remark}
  	\rm {
  	 	Condition (2) of Corollary \ref{b1} implies that $A_{V}\circ A_{W}=A_{W}\circ A_{V}$. In fact, by a simple calculation, while considering Lemma \ref{lemm}, we get $A_{V}\circ A_{W}=a^{2}(A^{*}_{\xi}\circ A^{*}_{\xi})-ab(A_{\xi}^{*}\circ A_{N}-A_{N}\circ A_{\xi}^{*})-b^{2}(A_{N}\circ A_{N})$. In view of (2) of Corollary \ref{b1}, we have $A_{V}\circ A_{W}=a^{2}(A^{*}_{\xi}\circ A^{*}_{\xi})-b^{2}(A_{N}\circ A_{N})$. On the other hand, $A_{W}\circ A_{V}=a^{2}(A^{*}_{\xi}\circ A^{*}_{\xi})-b^{2}(A_{N}\circ A_{N})$, which, if compared with previous relation, proves the assertion. It then follows that the vanishing of $d\tau$ on a leaf $M^{*}$ implies simultaneous diagonalisation of $A_{V}$, for all $V\in \Gamma(TM^{*\perp})$.
  	 	}
  	 \end{remark}

 	\noindent Next, we prove the following result.
 	 
 	 \begin{theorem}\label{theorem1}
 	    Let $(M,g)$ be a screen integrable null hypersurface of $\mathbb{R}_{1}^{n+2}$, and $f:M^{*}\longrightarrow \mathbb{R}_{1}^{n+2}$ an isometric immersion of a leaf $M^{*}$ of $S(TM)$ as a codimension 2 submanifold of $\mathbb{R}_{1}^{n+2}$. Suppose there exist a nonzero normal vector field $V$ to $M^{*}$ in $\mathbb{R}_{1}^{n+2}$ such that $d\tau=0$ on $S(TM)$ and $A_{V}^{*}=\lambda I$, $\lambda\ne 0$, then $f(M^{*})$ is contained inside 
 	 	\begin{enumerate}
 	 		\item $\mathbb{S}_{1}^{n+1}\left(\frac{\sqrt{\epsilon}}{\lambda}\right)$,\;\;\;if \;\; $\epsilon>0$,
 	 		\item $\mathbb{H}_{0}^{n+1}\left(\frac{-\sqrt{-\epsilon}}{\lambda}\right)$, if $\epsilon<0$,
 	 		
 	 	\end{enumerate}
 	 	where $\epsilon=\overline{g}(V,V)$. Furthermore, the null hypersurface $(M,g)$ is proper  quasi-screen conformal in $\mathbb{R}_{1}^{n+2}$. Moreover, if $A_{W}=0$, where $V,W$ spans the normal bundle $TM^{*\perp}$ then,  $(M,g)$ is a proper totally umbilic and screen totally umbilic in $\mathbb{R}_{1}^{n+2}$.
 	 \end{theorem}
 	 
 	 \begin{proof}
 	 	Observe that the vector $f(x)+\frac{1}{\lambda}V$ is constant for all $x\in M^{*}$. Let us denote it by $\tilde{c}$, then we have $\overline{g}(f(x)-\tilde{c},f(x)-\tilde{c})=\frac{1}{\lambda^{2}}\overline{g}(V,V)=\frac{\epsilon}{\lambda^{2}}$. As $d\tau=0$ on $S(TM)$, then $TM^{*\perp}$ is parallel by Corollary \ref{b1}. Consequently, $V$ is parallel and therefore, $f(M^{*})$ is contained in the sphere or  hyperbolic space with center $\tilde{c}$ by \cite{megid}. This proves parts (1) and (2). Furthermore, the condition  $A_{V}=\lambda I$ together with Lemma \ref{lemm} implies that $A_{V}=aA_{\xi}^{*}+bA_{N}=\lambda I$, where $V=a\xi +bN$.  Then, in view of \cite{no}, the null hypersurface $(M,g)$  is quasi-screen conformal in $\mathbb{R}_{1}^{n+2}$. On the other hand, if $A_{W}=0$, we have $aA^{*}_{\xi}-bA_{N}=0$. Combining this relation with the previous one gives $A^{*}_{\xi}=\frac{\lambda}{2a}I$ and $A_{N}=\frac{\lambda}{2b}I$, which shows that $(M,g)$ is totally umbilic and screen totally umbilic in $\mathbb{R}_{1}^{n+2}$, and the theorem is proved.
 	 	 	 \end{proof}
 	 
 	 \begin{corollary}\label{cor}
 	 	In case the vector field $V$, in Theorem \ref{theorem1}, is the mean curvature vector of $M^{*}$ in $\mathbb{R}_{1}^{n+2}$, then $M^{*}$ is immersed minimally in $\mathbb{S}_{1}^{n+1}$ or $\mathbb{H}_{0}^{n+1}$ (Such an immersion is called pseudo umbilic by Chen and Yano \cite{chen}). Moreover, the underlying null hypersurface  $(M,g)$ is a proper totally umbilic,  screen totally umbilic  and screen conformal in $\mathbb{R}_{1}^{n+1}$. 	 	
 	 \end{corollary}
 	 \begin{proof}
 	 	If $V$ is the mean curvature vector of $M^{*}$ in $\mathbb{R}_{1}^{n+2}$, then $V$ is parallel to the position vector $f(x)-\tilde{c}$ and therefore, by \cite{megid}, $M^{*}$ is minimal in either $\mathbb{S}_{1}^{n+1}$ or $\mathbb{H}_{0}^{n+1}$. In view of (\ref{t2}) and the fact that $f(M^{*})$ is pseudo umbilic, we have $C(X,Y)\xi +B(X,Y)N=g(X,Y)V=ag(X,Y)\xi+bg(X,Y)N$, for all $X,Y\in \Gamma(TM^{*})$. Taking the $\overline{g}$-product of the previous relation with $\xi$ and $N$, in turns, we get $B(X,Y)=bg(X,Y)$ and $C(X,Y)=ag(X,Y)$, respectively. Thus, $(M,g)$ is totally umbilic and screen totally umbilic in  $\mathbb{R}_{1}^{n+1}$. As $a,b\ne 0$, we deduce that $C(X,Y)=\psi B(X,Y)$, with $\psi=\frac{a}{b}$, showing that $M$ is screen conformal, which completes the proof. 	
 	  \end{proof}
 	 
 	\noindent As an example, we have the following. 
 	 
 	 \begin{example}[The null cone $\Lambda_{0}^{n+1}$ of $\mathbb{R}_{1}^{n+2}$]\label{exam}
 	 \rm{ 
 	 	Let $\mathbb{R}_{1}^{n+2}$ be the space $\mathbb{R}^{n+2}$ endowed with the semi-Euclidean metric
 	 	\begin{align}
 	 		\overline{g}(x,y)=-x^{0}y^{0}+\sum_{a=1}^{n+1}x^{a}y^{a},\;\;\;\mbox{where}\;\;\; x=\sum_{A=0}^{n+1}x^{A}\frac{\partial}{\partial x^{A}}.\nonumber
 	 	\end{align}
 	 	The null cone $\Lambda_{0}^{n+1}$ is given by the equation $-(x^{0})^{2}+\sum_{a=1}^{n+1}(x^{a})^{2}=0$, $x^{0}\ne 0$. It is known that $\Lambda_{0}^{n+1}$ is a null hypersurface of $\mathbb{R}_{1}^{n+2}$ and the radical distribution is spanned by a global vector field 
 	 	\begin{align}
 	 		\xi=\sum_{A=0}^{n+1}x^{A}\frac{\partial}{\partial x^{A}},
  	 	\end{align} 
 	 	on $\Lambda_{0}^{n+1}$. The unique section $N$ spanning the transversal bundle $tr(T\Lambda_{0}^{n+1})$  is given by 
 	 	\begin{align}
 	 		N=\frac{1}{2(x^{0})^{2}}\left\{-x^{0}\frac{\partial}{\partial x^{0}}+\sum_{a=1}^{n+1}x^{a}\frac{\partial}{\partial x^{a}} \right\},
 	 	\end{align}	 	
 	 	and is also globally defined. As $\xi$ is the position vector field we get 
 	 	\begin{align}
 	 		\overline{\nabla}_{X}\xi=\nabla_{X}\xi=X,\;\;\;\forall\, X\in \Gamma(T\Lambda_{0}^{n+1}).
 	 	\end{align} 	 	
 	 	Then, $A^{*}_{\xi}X+\tau(X)\xi+X=0$. Since $A^{*}_{\xi}$ is $\Gamma(S(T\Lambda_{0}^{n+1}))$-valued, we have 
 	 	\begin{align}\label{h3}
 	 		A_{\xi}^{*}X=-PX\;\;\;\mbox{and}\;\;\;\tau(X)=-\eta(X),\;\; \forall\, X\in \Gamma(T\Lambda_{0}^{n+1}).
 	 	\end{align}
 	 	Note that any  $X\in \Gamma(S(T\Lambda_{0}^{n+1}))$ is expressed as $X=\sum_{a=1}^{n+1}X^{a}\frac{\partial}{\partial x^{a}}$, where  $(X^{1},\ldots, X^{n+1})$ satisfy $\sum_{a=1}^{n+1}x^{a}X^{a}=0$. Then, 
 	 	\begin{align}\label{h2}
 	 		\nabla_{\xi}X=\overline{\nabla}_{\xi}X=\sum_{A=0,a=1}^{n+1}x^{A}\frac{\partial X^{a}}{\partial x^{A}}\frac{\partial}{\partial x^{a}},
 	 	\end{align}
 	 	from which we obtain 
 	 	\begin{align}\label{h1}
 	 		g(\nabla_{\xi}X,\xi)=\sum_{A=0,a=1}^{n+1}x^{a}X^{A}\frac{\partial X^{a}}{\partial x^{A}}=-\sum_{a=1}^{n+1}x^{a}X^{a}=0.
 	 	\end{align}
 	 	From (\ref{h2}) and (\ref{h1}), we have $\nabla_{\xi}X\in \Gamma(S(T\Lambda_{0}^{n+1}))$, that is, $A_{N}\xi=0$. Moreover, by simple calculations, we have 
 	 	\begin{align}\label{h4}
 	 		C(X,Y)=\overline{g}(\nabla_{X}Y,N)=\overline{g}(\overline{\nabla}_{X}Y,N)=-\frac{1}{2(x^{0})^{2}}g(X,Y).	 	
 	 	\end{align}
 	 	Clearly, $S(T\Lambda_{0}^{n+1})$ is integrable. Denote by $M^{*}$ its leaf, then 
 	 	\begin{align}\label{t1}
 	 		\overline{\nabla}_{X}Y=\nabla^{*}_{X}Y+\frac{g(X,Y)}{x^{0}}(-\frac{1}{2x^{0}}\xi-x^{0}N),\;\;\;\forall\, X,Y\in \Gamma(TM^{*}). 
 	 	\end{align} 
 	 	It is obvious that $M^{*}$ is a totally umbilic Riemannian submanifold of codimension 2 of $\mathbb{R}_{1}^{n+2}$. Moreover, using (\ref{h3}) and (\ref{h4}), we have 
 	 	\begin{align}
 	 		d\tau(X,Y)=\frac{1}{2}\{C(X,PY)-C(Y,PX)\}=0,\;\;\;\forall\, X,Y\in \Gamma(T\Lambda_{0}^{n+1}). 	 	\end{align}

 	 	 As $x^{0}\ne 0$, we may suppose $x^{0}>0$ (for $x^{0}<0$ we proceed analogously), and consider in the normal bundle $TM^{*\perp}$, the vector fields
 	 	 \begin{align}
 	 	 	V_{1}=-\frac{1}{2x^{0}}\xi-x^{0}N\;\;\;\mbox{and}\;\;\; V_{2}=-\frac{1}{2x^{0}}\xi+x^{0}N.	 	 
 	 	 \end{align}
 	 	 Note that $\{V_{1},V_{2}\}$ is an orthonormal basis, where $V_{1}$ and $V_{2}$ are spacelike and timelike, respectively. Using the expressions of $A_{\xi}^{*}$ and $A_{N}$, we get $A_{V_{1}}=\frac{1}{x^{0}}I$, from which $\lambda=\frac{1}{x^{0}}$. From the expressions of $\xi$ and $N$, we have $\overline{\nabla}_{X}V_{1}=-\frac{1}{x^{{0}}}X$ and $\overline{\nabla}_{X}V_{2}=0$, for all $X\in \Gamma(TM^{*})$. Therefore, from the Weingarten formula (\ref{t3}) for $M^{*}$, we get $\nabla^{*\perp}_{X}V_{1}=0$ and $\nabla^{*\perp}_{X}V_{2}=0$. Clearly, $\{V_{1},V_{2}\}$ is a parallel basis with respect to the normal connection $\nabla^{*\perp}$ of $M^{*}$. As the vector field $V_{1}$ is spacelike, that is, $\overline{g}(V_{1},V_{1})=1$ and also parallel to the mean curvature of $M^{*}$ in $\mathbb{R}_{1}^{n+2}$ (see (\ref{t1})), we conclude by Corollary \ref{cor} that $M^{*}$ is minimally immersed in the sphere $\mathbb{S}_{1}^{n+1}(x^{0})$. Note also that $(M,g)$ is totally umbilic, screen totally umbilic and screen conformal in $\mathbb{R}_{1}^{n+2}$. 	 
 	 }	
 	 \end{example}
 	 
 	 A subbundle $\mathcal{D}$ of the normal bundle  $TM^{*\perp}$ is said to be {\it parallel} in the normal bundle if it is invariant by parallel translation with respect to the normal connection $\nabla^{*\perp}$; that is, if $V\in \Gamma(\mathcal{D})$ then $\nabla_{X}^{*\perp}V\in \Gamma(\mathcal{D})$, for any $X\in \Gamma(TM^{*})$. We also say that the curvature tensor $R^{*\perp}$ of the normal connection $\nabla^{*\perp}$ is {\it parallel} in the normal bundle if $\nabla^{*\perp}R^{*\perp}=0$; that is, for any $X,Y,Z\in \Gamma(TM^{*})$	 and $V\in \Gamma(TM^{*\perp})$ , we have 
 	 \begin{align}\label{u1}
 	 	(\nabla^{*\perp}_{Z}R^{*\perp}(X,Y))V=\nabla^{*\perp}_{Z}R^{*\perp}(X,Y)V-R^{*\perp}(X,Y)\nabla^{*\perp}	_{Z}V=0.
 	  \end{align}
 	  	As an example, the curvature tensor $R^{*\perp}$  of the normal bundle $TM^{*\perp}$ of Example \ref{exam} is parallel in the normal bundle. This is due to the fact that the normal bundle is parallel; that is $\nabla^{*\perp}V=0$, for any $V\in \Gamma(TM^{*\perp})$. 
 	  	
 	  	Next, we define the {\it first normal space} $\mathcal{Q}(x)$ at $x\in M^{*}$ as the orthogonal complement in $T_{x}M^{*\perp}$ of $\{V(x)\in T_{x}M^{*\perp}:A_{V(x)}=0\}$.
 	 \begin{lemma}\label{lemm1}
 	 	Let $(M,g)$ be a screen integrable null hypersurface of $\overline{M}(c)$. Assume that $f:M^{*}\longrightarrow \overline{M}$ is an immersion of a leaf $M^{*}$ of $S(TM)$ as a codimension 2 submanifold of $\overline{M}$. Suppose that the curvature tensor $R^{*\perp}$ of the normal bundle to $M^{*}$ is parallel in the normal bundle. For each $x\in M^{*}$ let $\mathcal{D}(x)=\{V(x)\in T_{x}M^{*\perp}: R^{*\perp}(X,Y)V=0,\forall\, X,Y\}$. Then, $\mathcal{D}$ is parallel in the normal bundle $TM^{*\perp}$. Moreover, $d\tau$ vanishes on $TM^{*}$.
 	 \end{lemma}
 	 \begin{proof}
 	 	A proof of the first assertion follows similar steps as in \cite{docamo}. The vanishing of $d\tau$ on $TM^{*}$ follows from Proposition \ref{pop1} and the definition of $\mathcal{D}(x)$.
 	 \end{proof}
 	 \noindent Next, we proof the following result.
 	 \begin{theorem}\label{theorem2}
 	 Let $(M,g)$ be an $(n+1)$-dimensional screen integrable null hypersurface of a Lorentzian manifold $\overline{M}(c)$. Assume that $f:M^{*}\longrightarrow \overline{M}$ is a minimal, nontotally geodesic, immersion of a leaf $M^{*}$ of $S(TM)$ as a codimension 2 submanifold of $\overline{M}$.  Suppose that the curvature tensor $R^{*\perp}$ of the normal bundle to $M^{*}$ is parallel in the normal bundle. Then there exists an $(n+1)$-dimensional totally geodesic submanifold $M'$ of $\overline{M}$ such that $f$ is a minimal immersion of $M^{*}$ in $M'$. Furthermore, the underlying null hypersurface $(M,g)$ is minimal in $\overline{M}$ and its  shape operator $A_{N}$ satisfies $\mathrm{trace}_{|_{S(TM)}}A_{N}=0$.
 	 	 \end{theorem}
 	 \begin{proof}
 	 First note that; as $f$ is nontotally geodesic, then the first normal space of $f$ has constant dimension $1$. 	 	We first prove the case when the normal bundle $TM^{*\perp}$ is parallel. To that end,  let $\mathcal{Q}(x)$ be the first normal space at $x$. As $\dim \mathcal{Q}$ is constant, $\mathcal{P}=\mathcal{Q}^{\perp}$, where $\perp$ is the orthogonal complement in the normal bundle $TM^{*\perp}$, is a subbundle of the normal bundle. We want to show that $\mathcal{P}$ is parallel in the normal bundle and then, use a result of \cite{megid} to draw conclusions. Given $x\in M^{*}$, choose a unit vector field $V_{1}$, spanning $\mathcal{Q}$ at each point in a neighborhood $\mathscr{U}$ of $x\in M^{*}$. Let us extent the above field to $\{V_{1},V_{2}\}$ so that the latter span the normal space at a point of $\mathscr{U}$.  Consider the vector field $V_{2}$ which generate the subbundle $\mathcal{P}$ over $\mathscr{U}$. It then suffices to show that $\mathcal{P}$ is parallel in the normal bundle $TM^{*\perp}$. 
 	 	
 	 	Given $y\in \mathscr{U}$, let $X_{1},\ldots, X_{n}$ be coordinate vector fields in a neighborhood $\mathscr{U}'\subset \mathscr{U}$ of $y$ which diagonalize, at $y$, all the shape operators of $M^{*}$. This is possible because the normal bundle is flat, and it is a result of Cartan that the normal bundle is flat if and only if at each point all the second fundament forms are simultaneously diagonalizable. As $\overline{M}$ is a space of constant curvature $c$, we have $\overline{R}(X_{j},X_{i})V_{\alpha}=0$, for $1\le i,j\le n$ and $1\le \alpha\le 2$. Then, by (\ref{t5}), we have 
 	 	\begin{align}
 	 	 &-\nabla^{*}_{X_{j}}A_{V_{\alpha}}X_{i}+\nabla_{X_{i}}^{*}A_{V_{\alpha}}X_{j}-A_{\nabla^{*\perp}_{X_{i}}V_{\alpha}}X_{j}+A_{\nabla^{*\perp}_{X_{j}}V_{\alpha}}X_{i}+A_{V_{\alpha}}[X_{j},X_{i}]\nonumber\\
 		&+R^{*\perp}(X_{j},X_{i})V_{\alpha}+	\{C (X_{i},A_{V_{\alpha}}X_{j})-C(X_{j},A_{V_{\alpha}}X_{i})\}\xi \nonumber\\
 		&+\{B(X_{i},A_{V_{\alpha}}X_{j})-B(X_{j},A_{V_{\alpha}}X_{i})\}N=0,\;\;\;\forall\, X_{i},X_{j}\in T_{x}M^{*}.\label{i1}
 	 	\end{align}
 	 	For $V_{\alpha}\in \mathcal{P}$, that is $\alpha=2$, (\ref{i1}) gives 
 	 	\begin{align}
 	 		R^{*\perp}(X_{j},X_{i})V_{2}-A_{\nabla^{*\perp}_{X_{i}}V_{2}}X_{j}+A_{\nabla^{*\perp}_{X_{j}}V_{2}}X_{i}=0\nonumber,
 	 	 	\end{align} 
 	 	 	from which we get 
 	 	 	\begin{align}
 	 	 		-A_{\nabla^{*\perp}_{X_{i}}V_{2}}X_{j}+A_{\nabla^{*\perp}_{X_{j}}V_{2}}X_{i}=0,\label{i2}		 	
 	 	 \end{align}
 	 	 after considering the assumption $R^{*\perp}=0$. Since $X_{i},\ldots,X_{n}$ diagonalize all the shape operators at $y$, we have, at $y$, $A_{\nabla^{*\perp}_{X_{j}}V_{2}}X_{i}=a_{j}^{i}X_{i}$ and $A_{\nabla^{*\perp}_{X_{i}}V_{2}}X_{j}=b_{i}^{j}X_{j}$, for some numbers $a_{j}^{i}$, $b_{i}^{j}$. Therefore, (\ref{i2}) implies that $A_{\nabla^{*\perp}_{X_{j}(y)}V_{2}}X_{i}(y)=0$, for all $i\ne j$. As the immersion is minimal, we have $A_{\nabla^{*\perp}_{X_{j}(y)}V_{2}}X_{i}(y)=0$, for all $i,j$, and therefore, $A_{\nabla^{*\perp}_{X_{j}(y)}V_{2}}=0$ for all $j$. This in turn implies that $\nabla^{*\perp}_{X_{j}(y)}V_{2}\in \mathcal{P}(y)$, following the definition of $\mathcal{P}$.  That is, $\mathcal{P}$ over $\mathscr{U}$ is parallel in the normal bundle. This implies that $\mathcal{Q}$ is also parallel in the normal bundle. In fact, take $V_{2}\in \mathcal{P}$ and $V_{1}\in \mathcal{Q}$. As $\nabla^{*\perp}$ is a metric connection, we have $X\overline{g}(V_{2},V_{1})=\overline{g}(\nabla_{X}^{*\perp}V_{2},V_{1})+\overline{g}(V_{2},\nabla_{X}^{*\perp}V_{1})=0$, from which we see that $\mathcal{Q}$ is also parallel in the normal bundle. By Theorem 0.2 of \cite[p. 33]{megid}, there exist a totally geodesic submanifold $M'$ of $\overline{M}$ such that $\dim M'=n+1$ and $f(M^{*})\subset M'$. 
 	 	 
 	 	 Next,  since $\dim \mathcal{Q}=1$ we have $\dim \mathcal{P}=1$. Let $V_{2}\in \mathcal{P}$, then $A_{V_{2}}=0$ by the definition of $\mathcal{P}$. Thus, Lemma \ref{lemm} gives $aA^{*}_{\xi}+bA_{N}=0$. Taking the trace of this relation along $TM^{*}$, we get $a\, \mathrm{trace}A^{*}_{\xi}+b\,\mathrm{trace}A_{N}=0$. On the other hand, since $M^{*}$ is minimal in $\overline{M}$, we have $(\mathrm{trace}A_{V_{1}})V_{1}+(\mathrm{trace}A_{V_{2}})V_{2}=0$, where $V_{1}\in \mathcal{Q}$. In view of Lemma \ref{lemm}, we have $a\, \mathrm{trace}A^{*}_{\xi}-b\,\mathrm{trace}A_{N}=0$. Therefore, solving gives $\mathrm{trace}A^{*}_{\xi}=0$ and $\mathrm{trace}A_{N}=0$, showing that $(M,g)$ is minimal in $\overline{M}$. This completes the first case of the proof.
 	 	 
 	 	 Turning to the second case, that is; the normal bundle is not flat. Set $\mathcal{D}(x)=\{V(x)\in T_{x}M^{*\perp}: R^{*\perp}(X,Y)V=0,\forall\, X,Y\}$. By Lemma \ref{lemm1},  $\mathcal{D}$ is parallel in the normal bundle. Let $\mathcal{P}$ be the orthogonal complement of $\mathcal{Q}$ in the normal bundle $TM^{*\perp}$.  It is obvious that $\mathcal{P}\subset \mathcal{D}$. Observe that, by Lemma \ref{lemm1}, $\mathcal{D}$ is parallel and , by (\ref{t8}), all the shape operators $A_{V}$, $V\in \mathcal{D}$, can be simultaneously diagonalized. Therefore, we can apply the same arguments as in the first case above, with $\mathcal{D}$ in place of $TM^{*\perp}$, to conclude that $\mathcal{P}$, hence $\mathcal{Q}$, is parallel in the normal bundle. Again, as in the first case, by Megid's theorem 0.2 in \cite{megid}, $f(M^{*})\subset M'$, where $M'$ is a totally geodesic submanifold of $\overline{M}$ with dimension $n+1$. The minimality of $(M,g)$, as a null hypersurface of $\overline{M}$, also follows as in the previous case. This completes the second case and so the theorem is proved.
 	 \end{proof}


\end{document}